\documentclass[11pt, oneside]{article}   	
\usepackage{fullpage}

\usepackage[utf8]{inputenc}
\usepackage[T1]{fontenc}
\usepackage{lmodern}
\usepackage{graphicx}
\usepackage{setspace}
\usepackage[english]{babel}
\usepackage{lipsum}
\usepackage{fancyhdr}
\usepackage{color}			
\usepackage{amsmath,amsfonts,amsthm,amssymb,mathrsfs,bm,mathtools,nicefrac}
\usepackage{url}
\usepackage[shortlabels]{enumitem}

\newcommand{\de}{\partial}
\newcommand{\di}{\text{d}}

\newcommand{\norm}[1]{\left\Vert #1\right\Vert}

\newcommand{\NN}{\mathbb{N}}
\newcommand{\ZZ}{\mathbb{Z}}
\newcommand{\RR}{\mathbb{R}}

\newcommand{\TT}{\mathbb{T}}
\newcommand{\PP}{\mathbb{P}}

\newtheorem{theorem}{Theorem}[section]
\newtheorem{prop}[theorem]{Proposition}
\newtheorem{lemma}[theorem]{Lemma}

\theoremstyle{remark}
\newtheorem{rmk}[theorem]{Remark}

\title{Enhanced Dissipation and Transition Threshold for the Poiseuille Flow in a Periodic Strip}
\author{Augusto Del Zotto\footnote{Department of Mathematics, Imperial College London, SW7 2AZ, UK. 
E-mail: a.del-zotto20@imperial.ac.uk}}

\begin{document}
\maketitle
\abstract{We consider the solution to the 2D Navier-Stokes equations around the Poiseuille flow $(y^2,0)$ on $\TT\times\RR$ with small viscosity $\nu>0$.  Via a hypocoercivity argument, we prove that the $x-$dependent modes of the solution to the linear problem undergo  the enhanced dissipation effect with a rate proportional to $\nu^{\frac{1}{2}}$.  Moreover, we study the nonlinear enhanced dissipation effect and we establish a transition threshold of $\nu^{\frac{2}{3}+}$. Namely, when the perturbation of the Poiseuille flow is size at most $\nu^{\frac{2}{3}+}$, its size remains so for all times and the enhanced dissipation persists with a rate proportional to $\nu^{\frac{1}{2}}$.  
}

\section{Introduction}
We study the 2D Navier-Stokes equations 
\begin{equation}\label{NS}
\left\{
\begin{array}{l}
\de_t U + (U\cdot \nabla)U+\nabla P-\nu \Delta U=0,\\
\nabla\cdot U=0,\\
\end{array}
\right.
\end{equation}
defined on the domain $\TT\times\RR$.  $U=(U_1,U_2)$ is the velocity vector field, $P$ is the scalar pressure and $\nu$ is the viscosity coefficient of the fluid, proportional to the inverse of the Reynolds number.
Defining the vorticity of $U$ as $\Omega=\nabla^\perp \cdot U$,  where $\nabla^\perp =(-\de_y,\de_x)$, it is possible to remove the pressure term and to rewrite the above system as
\begin{equation}\label{NSvort}
\left\{
\begin{array}{l}
\de_t \Omega + U\cdot \nabla\Omega-\nu \Delta \Omega=0,\\
\Omega=\Delta\Psi,\\
U=\nabla^\perp\Psi.
\end{array}
\right.
\end{equation}
Here $\Psi$ is the corresponding stream function for the vector field $U$.
It is easy to see that the so called Poiseuille Flow $$U_P=(y^2,0), \qquad   \Omega_P=-2y,$$
is a stationary solution of \eqref{NSvort}.
To study the dynamic near the Poiseuille flow we consider a small perturbation of it. We set $U=U_P+u$,  so $\Omega=\Omega_P + \omega$,  where $u,\omega$ are the perturbations for the velocity field and the vorticity, respectively. The vorticity formulation for $\omega$ reads as follow
\begin{equation}\label{IVP Poiseuille}
\left\{
\begin{array}{l}
\de_t \omega + y^2\de_x\omega-2\de_x\psi-\nu\Delta\omega=-u\cdot\nabla\omega ,\\
\omega =\Delta \psi,\\
u=\nabla^\perp\psi .
\end{array}
\right.
\end{equation}
The boundary conditions here are set to be periodic for the $x$ variable, while for the $y$ direction $\omega$ is assumed to have sufficient decay at infinity. The stability of the Poiseuille flow can then be seen as the decay of the solution to \eqref{IVP Poiseuille}.  An overview of the paper is presented in the following subsections of the introduction. The first part is devoted to the analysis of the linear problem while the second part focuses on the transition threshold problem for the fully non linear equation.
\subsection{Linear enhanced dissipation and estimates for $\mathcal{L}_\nu$}
We define the linear operator 
\begin{equation}\label{linear operator}
\mathcal{L}_\nu = - y^2\de_x+2\de_x\Delta^{-1}+\nu\Delta
\end{equation}
associated to the linearized counterpart of \eqref{IVP Poiseuille}. 
In the first part, we establish some decay estimates for the semigroup generated by $\mathcal{L}_\nu$ in the usual $L^2$ norm. 
Our first result is the following,  here $\PP_{\neq}$ is the projection on the nonzero $x$-Fourier modes.
\begin{theorem}\label{thm hypo intro}
Let $\nu<1$ and $g\in L^2(\TT\times\RR)$. There exist constants $C_0,c_0>0$ independent of $\nu$ such that
\begin{equation}\label{linear enhanced dissipation}
\norm{e^{\mathcal{L}_\nu t }\PP_{\neq}(g)}_{L^2}\leq C_0e^{-c_0\nu^\frac{1}{2}t}\norm{\PP_{\neq} (g)}_{L^2},
\end{equation}
for all $t\geq 0.$
\end{theorem}
This result gives a quantitative estimate of the linear \emph{enhanced dissipation} effect.  Indeed, the timescale obtained for the nonzero modes of the initial datum is proportional to $\nu^{-1/2}$,  which is much faster than the heat equation timescale, proportional to $\nu^{-1}$.

We refer to enhanced dissipation as the phenomenon where the mixing properties of the fluid allow to improve the natural heat dissipation timescale $\mathcal{O}(\nu^{-1})$ to a faster timescale $\mathcal{O}(d(\nu)^{-1})$ that satisfies
\[\lim_{\nu\to 0}\frac{\nu}{d(\nu)}=0.\]
This phenomenon has been widely studied in the physics literature, see for example \cite{DR1981,R1879,R1883} and mathematics literature \cite{CKRZ2008,CZDE2020}. 
In the context of the Navier-Stokes equations near shear flows, we cite results for the well known Couette flow \cite{BMV2016,wei2018transition}, with a dissipation timescale $\mathcal{O}(\nu^{-1/3})$, and the Kolmogorov flow \cite{IMM2019pseudospectral,WZ2019,WZZ2020}, where the rate is known to be $\mathcal{O}(\nu^{-1/2})$. Regarding the Poiseuille flow, the first linear enhanced dissipation result was given by Coti Zelati, Elgindi and Widmayer \cite{CZEW2020} for the unbounded 2D domain $(x,y)\in\TT\times\RR$. In their paper the linear enhanced dissipation effect is established  around the Poiseuille flow. The rate obtained is proportional to $\nu^{1/2}(1+|\log\nu|)^{-1}$ in the weighted $L^2$ space with norm
$$\norm{f}^2_X=\norm{f}^2_{L^2}+\norm{yf}_{L^2}^2.$$ Ding and Lin \cite{DL2020} proved the same decay rate,  without the logarithmic correction,  for the Poiseuille flow in a bounded 2D channel $\TT\times [-1,1].$ We cite also the paper by Chen, Wei and Zhang \cite{CWZ2019}, where a $ \mathcal{O}(\nu^{-1/2})$ rate is obtained for the 3D pipe Poiseuille flow.  The approach used in the last two papers is completely different from the hypocoercivity method.  It relies on resolvent estimates and a Gearhart-Pr\"{u}ss type Theorem introduced by Wei in \cite{W2021}.

Our first result is a sharpening of \cite{CZEW2020},  indeed we are able to remove the logarithmic correction and to get a decay rate of $\mathcal{O}(\nu^{-1/2})$ in $L^2(\TT\times\RR)$.  
The proof of Theorem  \ref{thm hypo intro} relies on a hypocoercivity argument \cite{V2009hypo}, similar to the one in \cite{CZEW2020}.  Here, we construct an energy functional following Wei and Zhang  idea \cite{WZ2019},  namely, each term of the functional has a time dependent weight.  Theorem \ref{thm hypo intro} follows then from an iteration argument.  Furthermore, thanks to the time dependent weights of the energy functional, we are able to prove additional estimates on the semigroup $e^{\mathcal{L}_\nu t}$ generated by the linearized operator $\mathcal{L}_\nu$.  These estimates play a crucial role in establishing the transition threshold.

\subsection{Nonlinear enhanced dissipation and transition threshold}
Our second result concerns the transition threshold for the 2D Poiseuille flow.  The asymptotic stability of fluid motion between parallel plates was firstly analyzed by Kelvin \cite{K1887}, who introduced the following concept: the stability may depend on the viscosity coefficient $\nu$ in such a way that the stability threshold decreases whenever $\nu$ decreases. The mathematical formulation of this problem can be given as follows.
Given a norm $\norm{\cdot}_X$, find a $\gamma=\gamma(X)$ such that
\begin{equation*}
\begin{array}{ccc}
\norm{u}_X\leq \nu^\gamma & \Rightarrow & stability,\\
\norm{u}_X\gg \nu^\gamma & \Rightarrow & instability,
\end{array}
\end{equation*}
here $u$ is a perturbation of the flow.

The transition threshold problem for the Couette flow has been deeply studied recently \cite{BGM2015above,BGM2017,BGM2015below,BVW2018}, in both the frameworks of Sobolev spaces and Gevrey classes.  In the 2D case the transition threshold is known to be $\gamma\leq \frac{1}{3}$, \cite{MZ2019couette}. 
For the 2D Kolmogorov flow in the periodic box, it holds that $\gamma\leq\frac{2}{3}+\varepsilon$ for any $\varepsilon>0$,  see \cite{WZZ2020}.
For the Poiseuille flow, it has been proved by Coti Zelati, Elgindi and Widmayer \cite{CZEW2020} that $\gamma\leq \frac{3}{4}+\varepsilon$ in $\TT\times \RR$ and by Ding and Lin \cite{DL2020} that $\gamma\leq \frac{3}{4}$ in $\TT\times[-1,1]$ with Navier-slip boundary conditions, i.e. $\omega(\pm 1)=\psi(\pm 1)=0$.

In the second part of this paper we are going to show a transition threshold for the Poiseuille Flow on $\TT\times\RR$ with  $\gamma\leq \frac{2}{3}+\varepsilon.$
We are able to prove the following Theorem using the linear enhanced dissipation and the estimates on the semigroup generated by the linearized operator.
\begin{theorem}
 There exists constants $\varepsilon_0\in (0,1), C_1>0,c_1>0$ such that for all $0<\nu<1$ and for every $\omega_{in}\in L^2$ with $$\norm{\omega_{in}}_{L^2}\leq \varepsilon_0(1+|\log\nu|^\frac{1}{2})^{-1}\nu^{2/3},$$ the solution $\omega $ of \eqref{IVP Poiseuille} is global in time with the bound $$\norm{\PP_{\neq}(\omega)(t)}_{L^2} \leq C_1e^{-c_1\nu^{1/2}t}\norm{\PP_{\neq}(\omega_{in})}_{L^2}.$$
\end{theorem}
The proof is based on careful estimates of the non-zero modes of the nonlinear term in \eqref{IVP Poiseuille}.
We remark that our theorem gives a better transition threshold for the planar Poiseuille flow in $\TT\times\RR$, bringing it from $\nu^{\frac{3}{4}+}$ to $\nu^{\frac{2}{3}+}$, more precisely $(1+|\log\nu|^\frac{1}{2})^{-1}\nu^{\frac{2}{3}}$. 

\paragraph{Structure of the paper}
In Section 2 we define the modified energy functional \eqref{functional} and we prove Theorem \ref{thm hypo}. We then deduce the enhanced dissipation for the linear problem \eqref{IVP linearized}.  Section 3 is devoted to prove Lemma \ref{main lemma}, which establishes additional estimates on the semigroup generated by the linearized operator. Section 4 concludes the paper and contains the proof of Theorem \ref{transition thm}.
\paragraph{Notation}
Throughout this paper we use 
$$\nabla_k=\PP_k\nabla=(ik,\de_y)$$ for the projected gradient onto the $\pm k$-th $x$-Fourier modes and
$$\Delta_k=\PP_k\Delta=-k^2+\de_y^2$$
for the projected Laplacian. 
Moreover, we use $C>0$ to indicate a constant independent of $\nu, k $ and $t$ that may vary line by line.  We also denote $\norm{\cdot}=\norm{\cdot}_{L^2}$ the usual $L^2$ norm. 

\section{Hypocoercivity Estimates}
For a $x$-periodic function $f$ we write its Fourier expansion as
\begin{equation}
f(t,x,y)=\sum_{j\in\ZZ}a_j(t,y)e^{ijx}, \qquad a_j(t,y)=\frac{1}{2\pi}\int_\TT f(t,x,y)e^{-ijx}\di x.
\end{equation}
For $k\in \NN_0$ we set 
\begin{equation}
f_k(t,x,y)=\sum_{|j|=k}a_j(t,y)e^{ijx},
\end{equation}
so that 
\begin{equation}
f(t,x,y)=\sum_{k\in\NN_0}f_k(t,x,y)
\end{equation}
can be expressed as a sum of real valued functions localized in $x$-frequency on a single band $\pm k$, $k\in\NN_0$.  We also introduce the following operators: given a function $f$ we define
\begin{equation}
\PP_0(f)=\frac{1}{2\pi}\int_\TT f(x,y)\di x, \qquad \PP_{\neq}(f)=f-\PP_0(f),
\end{equation}
and for any $k\in \NN_0$, we denote $\PP_k$ the projection to the sum of the $\pm k$-th Fourier modes in $x$.  
We start by considering the linearized system associated to \eqref{IVP Poiseuille} with initial datum $g\in L^2$, i.e. 
\begin{equation}\label{IVP linearized}
\left\{
\begin{array}{l}
\de_t \omega + y^2\de_x\omega-2\de_x\psi-\nu\Delta\omega=0 ,\\
\Delta \psi =\omega,\\
\omega|_{t=0}=g.
\end{array}
\right.
\end{equation}
The solution is given by $\omega(t)=e^{\mathcal{L}_\nu t}g$ for all $t\geq 0,$ where $\mathcal{L}_\nu$ is the linear operator \eqref{linear operator}.
To establish decay estimates for the semigroup we proceed by 
defining the following energy functional,  where $\alpha, \beta, \gamma$ have to be determined,
\begin{equation}\label{functional}
\Phi(t)=\frac{1}{2}\norm{\omega}^2+\frac{1}{2}\alpha \nu t \norm{\nabla\omega}^2+2\beta\nu t^2\langle \de_y\omega,y\de_x\omega\rangle +\frac{1}{2}\gamma \nu t^3 \left[\norm{y\de_x\omega}^2+2\norm{\nabla \de_x\psi}^2\right].
\end{equation}
This energy functional resembles the one used in \cite{CZEW2020} and differs from it by the time dependent weights. This modification is the key to achieve enhanced dissipation in the $L^2$ norm and to remove the $1+|\log\nu|$ correction obtained in \cite{CZEW2020}.
\begin{theorem}\label{thm hypo}
Fix $0<\varepsilon<\frac{1}{36}$ and define $\alpha=\varepsilon^2,\beta=\varepsilon^3, \gamma=16\varepsilon^4$. Then
\begin{equation}\label{condizioni}
\alpha<\frac{1}{2}, \quad 16\beta^2\leq \alpha\gamma, \quad 2\alpha^2<\beta<\frac{1}{4}\alpha, \quad 9\gamma<4\beta,
\end{equation}
and it follows that $\Phi(t)$ satisfies 
\begin{equation}
\Phi(t) \geq \frac{1}{2}\norm{\omega}^2+\frac{1}{4}\alpha \nu t \norm{\nabla\omega}^2 +\frac{1}{4}\gamma \nu t^3 \left[\norm{y\de_x\omega}^2+2\norm{\nabla \de_x\psi}^2\right]
\end{equation}
and
\begin{equation}
\Phi'(t)\leq -\gamma \nu t^3 \norm{\de_x\omega}^2
\end{equation}
for every $t\geq 0.$
\end{theorem}
Once this statement is proved, linear enhanced dissipation follows directly. 
\begin{theorem}
Let $\nu<1$ and $g\in L^2$.  There exist constants $C_0,c_0>0$ independent of $\nu$ such that
\begin{equation}
\norm{e^{\mathcal{L}_\nu t }\PP_k(g)}\leq C_0e^{-c_0\nu^\frac{1}{2}|k|^\frac{1}{2}t}\norm{\PP_k(g)},
\end{equation}
for all $t\geq 0.$ 
Combining all non-zero modes together we get
\begin{equation}\label{enhanced diss}
\norm{e^{\mathcal{L}_\nu t }\PP_{\neq}(g)}\leq C_0e^{-c_0\nu^\frac{1}{2}t}\norm{\PP_{\neq}(g)}.
\end{equation}
In addiction,  for the non-zero modes of the velocity $u_k$, we have the following decay
\begin{equation}\label{stima velocità}
\norm{\PP_k(u)(t)}^2\leq \frac{2}{\gamma \nu k^2t^3}\norm{\PP_k(g)}^2,
 \end{equation}
 for all $t\geq 0.$
\end{theorem}
We briefly give the proof. To simplify the notation we will write $f_k$ instead of $\PP_k(f)$.
\begin{proof}[Proof of Theorem 2.2]
Since the equations decouple in $k$, we can apply separately Theorem \ref{thm hypo} to the $\pm k$-th $x$-frequency and get 
\begin{equation}
\Phi_k(t)\geq \frac{1}{2}\norm{e^{\mathcal{L}_\nu t}g_k}^2
\quad \text{and} \quad
\Phi'_k(t)\leq -\gamma\nu^2 k^2 t^3\norm{e^{\mathcal{L}_\nu t}g_k}^2.
\end{equation}
Then,
\begin{align*}
\frac{1}{2}\norm{e^{\mathcal{L}_\nu t}g_k}^2&\leq \Phi_k(t) \\
&= \Phi_k(0)+\int_0^t\Phi_k'(s)\di s \\
&\leq \frac{1}{2} \norm{g_k}^2-\int_0^t\gamma\nu^2 k^2 s^3\norm{e^{\mathcal{L}_\nu s}g_k}^2\di s\\
&\leq \frac{1}{2}\norm{g_k}^2-\frac{\gamma}{4}\nu^2k^2t^4\norm{e^{\mathcal{L}_\nu t}g_k}^2.
\end{align*}
Rearranging this inequality we obtain 
\begin{equation}\label{norm bound}
\norm{e^{\mathcal{L}_\nu t}g_k}^2\leq \frac{1}{1+\frac{\gamma}{2}\nu^2k^2t^4}\norm{g_k}^2.
\end{equation}
To prove the enhanced dissipation we proceed by iteration. Fix a time $t_0$ for which  $$\frac{1}{1+\frac{\gamma}{2}\nu^2k^2t_0^4}=	\frac{1}{2}.$$ 
Now, for every time $t>t_0$, write it as $t=\lfloor t_0^{-1}t\rfloor t_0 + t^*$, with $t^* \in [0,t_0) $. Using the fact that \eqref{IVP linearized} is autonomous, we deduce from \eqref{norm bound} that for every $s\in[0,t_0]$
\begin{equation}
\norm{e^{\mathcal{L}_\nu (t_0+s)}g_k}^2\leq \frac{1}{1+\frac{\gamma}{2}\nu^2k^2t_0^4}\norm{e^{\mathcal{L}_\nu s}g_k}^2 \leq \frac{1}{2}\norm{e^{\mathcal{L}_\nu s}g_k}^2.
\end{equation}
Then, using the fact that $t\mapsto \norm{e^{\mathcal{L}_\nu t}g_k}$ is decreasing, by iteration we have
\begin{align*}
\norm{e^{\mathcal{L}_\nu t}g_k}^2&\leq  \norm{e^{\mathcal{L}_\nu\lfloor t_0^{-1}t\rfloor t_0 }g_k}^2\\
&\leq  \left(\frac{1}{2}\right)^{\lfloor t_0^{-1}t\rfloor}\norm{g_k}^2\\
&\leq C_0e^{-c_0\nu^\frac{1}{2}|k|^\frac{1}{2}t}\norm{g_k}^2.
\end{align*}
Finally,  from the monotonicity of $\Phi(t)$ we can deduce \eqref{stima velocità}. Indeed, 
\begin{equation}\label{stima psi}
\frac{1}{2}\gamma\nu t^3\norm{\de_x\nabla\psi}^2\leq \norm{g}^2,
\end{equation}
and recalling that $$\norm{u_k}^2=\norm{\nabla_k^\perp\psi_k}^2=\norm{\nabla_k\psi_k}^2,$$
we have the last inequality.

\end{proof}
In order to prove Theorem \ref{thm hypo} we state some preliminary identities that will be used to compute the derivative of the functional $\Phi$.
\begin{prop}\label{derivat}
Let $\omega$ be a solution to \eqref{IVP linearized}. Then the following holds:
\begin{equation}
\frac{1}{2}\frac{\di }{\di t}\norm{\omega}^2=-\nu\norm{\nabla \omega}^2;
\end{equation}
\begin{equation}
\frac{1}{2}\frac{\di }{\di t}\norm{\nabla\omega}^2=-\nu\norm{\Delta\omega}^2-2\langle y\de_x\omega,\de_y\omega\rangle;
\end{equation}
\begin{equation}
\frac{\di }{\di t}\langle \de_y\omega,y\de_x\omega\rangle=-2\norm{y\de_x\omega}^2-4\norm{\de_{xy}\psi}^2-2\nu\langle\Delta\omega,y\de_{xy}\omega\rangle;
\end{equation}
\begin{equation}
\frac{1}{2}\frac{\di }{\di t}\left[\norm{y\de_x\omega}^2+2\norm{\nabla\de_x\psi}^2\right]=-\nu\norm{\de_x\omega}^2-\nu\norm{y\de_x\nabla\omega}^2.
\end{equation}
\end{prop}
The proof of these identities can be found in \cite[Lemma 2.4]{CZEW2020}.  We proceed with the proof of Theorem \ref{thm hypo}.

\begin{proof}[Proof of Theorem \ref{thm hypo}]
For sake of clarity we recall here that $\Phi(t)$ is defined as
\begin{equation}
\Phi(t)=\frac{1}{2}\norm{\omega}^2+\frac{1}{2}\alpha \nu t \norm{\nabla\omega}^2+2\beta\nu t^2\langle \de_y\omega,y\de_x\omega\rangle +\frac{1}{2}\gamma \nu t^3 \left[\norm{y\de_x\omega}^2+2\norm{\nabla \de_x\psi}^2\right].
\end{equation}
Using the Cauchy-Schwarz inequality and the Young inequality we get
\begin{equation}
2\beta t^2|\langle \de_y\omega,y\de_x\omega\rangle|\leq 2\beta t^2\norm{\de_y\omega}\norm{y\de_x\omega}\leq \frac{\alpha t}{4}\norm{\de_y\omega}^2+\frac{4\beta^2t^3}{\alpha}\norm{y\de_x\omega}^2.
\end{equation}
Plugging this into $\Phi(t)$ we get
\begin{equation}
\Phi(t)\geq \frac{1}{2}\norm{\omega}^2+\frac{1}{4}\alpha \nu t \norm{\nabla\omega}^2+\frac{1}{2}\left(\gamma-\frac{8\beta^2}{\alpha}\right) \nu t^3 \left[\norm{y\de_x\omega}^2+2\norm{\nabla \de_x\psi}^2\right],
\end{equation}
hence the lower bound.
We now proceed by computing the derivative of the functional. By Proposition \ref{derivat} we have
\begin{align*}
\frac{\di }{\di t}\Phi(t)&=-\nu\norm{\nabla w}^2+\frac{\alpha\nu}{2}\norm{\nabla w}^2 +\alpha\nu t\left(-\nu\norm{\Delta w}^2-2\langle y\de_x,\de_yw\rangle\right)+4\beta t \nu \langle \de_yw,y\de_xw\rangle \\
&\qquad+2\beta \nu t^2\left(-2\nu\langle\Delta w,y\de_x\de_yw\rangle-2\norm{y\de_xw}^2-4\norm{\de_x\de_y\psi}^2\right)\\
&\qquad +\frac{3}{2}\gamma t^2\nu\left(\norm{y\de_xw}^2+2\norm{\nabla\de_x\psi}^2\right)+\gamma \nu t^3 \left(-\nu\norm{y\de_x\nabla w}^2-\nu\norm{\de_xw}^2\right).
\end{align*}
After rearranging all the terms we obtain
\begin{equation}
\frac{\di }{\di t}\Phi(t) = I_1+I_2+I_3+I_4+I_5,
\end{equation}
where 
\begin{align*}
 I_1&=-\frac{\nu}{2}\norm{\nabla w}^2+\alpha\nu\norm{\nabla w}^2-\gamma \nu^2t^3\norm{\de_x\omega}^2 \\
I_2&= -\alpha \nu^2t\norm{\Delta\omega}^2-4\beta\nu^2t^2\langle\Delta\omega,y\de_{xy}\omega \rangle-\gamma\nu^2t^3\norm{y\de_x\nabla\omega}^2\\
I_3&=-\frac{\nu}{2}\norm{\nabla\omega}^2 -2\alpha\nu t\langle y\de_x,\de_yw\rangle-\beta\nu t^2\norm{y\de_x\omega}^2-2\beta\nu t^2\norm{\de_{xy}\psi}^2 \\
I_4&=-\frac{\nu\alpha}{2}\norm{\nabla\omega}^2 +4\beta\nu t\langle y\de_x,\de_yw\rangle -\beta\nu t^2\norm{y\de_x\omega}^2-2\beta\nu t^2\norm{\de_{xy}\psi}^2 \\
I_5&=-2\nu\beta t^2\left[\norm{y\de_x\omega}^2+2\norm{\de_{xy}\psi}^2 \right]+\frac{3}{2}\gamma\nu t^2\left[\norm{y\de_x\omega}^2+2\norm{\nabla\de_x\psi}^2\right]
\end{align*}
Recalling now the conditions \eqref{condizioni} on the constants $\alpha,\beta,\gamma$, namely 
$$\alpha<\frac{1}{2}, \quad 16\beta^2\leq \alpha\gamma, \quad 2\alpha^2<\beta<\frac{1}{4}\alpha, \quad 9\gamma<4\beta,$$
we use again the Cauchy-Schwarz inequality and the Young inequality to get
\begin{equation}
I_1\leq -\gamma\nu^2t^3\norm{\de_x\omega}^2;
\end{equation}
\begin{equation}
I_2 \leq -\frac{1}{2}\alpha \nu^2t\norm{\Delta\omega}^2 -\left(\gamma-\frac{8\beta^2}{\alpha}\right)\nu^2t^3\norm{y\nabla\de_x\omega}^2 <0;
\end{equation}
\begin{equation}
I_3 \leq -\frac{\nu}{4}\norm{\nabla\omega}^2 -(\beta-2\alpha^2)\nu t^2\norm{y\de_x\omega}^2-2\beta\nu t^2\norm{\de_{xy}\psi}^2 <0;
\end{equation}
\begin{equation}
I_4\leq -\frac{\nu}{2}(\alpha-4\beta)\norm{\nabla\omega}^2 -\frac{\beta}{2}\nu t^2\norm{y\de_x\omega}^2-2\beta\nu t^2\norm{\de_{xy}\psi}^2<0.
\end{equation}
Moreover,  from
\begin{equation}
 \langle \de_{xy}\psi,y\de_x\omega\rangle=-\frac{1}{2}\norm{\de_{xy}\psi}^2+\frac{1}{2}\norm{\de_{xx}\psi}^2
 \end{equation} 
we can deduce that
\begin{equation}
\norm{\nabla\de_x\psi}^2\leq \norm{y\de_x\omega}^2+3\norm{\de_{xy}\psi}^2,
\end{equation}
hence
\begin{equation}
I_5 \leq -\left(2\beta-\frac{9}{2}\gamma\right)\nu t^2\left[\norm{y\de_x\omega}^2+2\norm{\nabla\de_x\psi}^2\right]\leq 0.
\end{equation}
Combining all together we get the upper bound on $\Phi'(t)$.
\end{proof}
\section{Additional Linear Estimates}
In this section we give some estimates on the semigroup generated by the linearized operator $\mathcal{L}_\nu.$ These estimates play a crucial role in the estimates for the nonlinear term in the full perturbed system.
\begin{lemma}\label{main lemma}
Let $g\in L^2$ such that $\PP_0(g)=0$. The following estimates hold for every $T>0$:
\begin{equation}\label{primastima}
\int_0^T\norm{\nabla (e^{\mathcal{L}_\nu t}g)}^2\leq \frac{1}{2}\nu^{-1}\norm{g}^2;
\end{equation}
\begin{equation}\label{secondastima}
\int_0^T\norm{\de_x(e^{\mathcal{L}_\nu t}g)}\leq C \nu^{-2/3}\norm{g};
\end{equation}
\begin{equation}\label{terzastima}
\int_0^T\norm{\nabla\Delta^{-1}(e^{\mathcal{L}_\nu t}\PP_{\neq}(g))}^2_{L^\infty}\leq C(|\log\nu|+1)\nu^{-1/3}\norm{\PP_{\neq}(g)}^2.
\end{equation}

\end{lemma}
\begin{proof}
Throughout the proof, $T$ will be any positive time.  
The first estimate follows directly from the energy inequality for the linearized problem
\begin{equation}
\norm{e^{\mathcal{L}_\nu t}g}^2+2\nu\int_0^t\norm{\nabla (e^{\mathcal{L}_\nu t}g)}^2\leq \norm{g}^2
\end{equation}
and it hold for any $g\in L^2$.
For the second estimate we note that,  for any $k\neq 0$,
\begin{equation}
\left\{
\begin{array}{l}
\norm{\de_x(e^{\mathcal{L}_\nu t}g)}^2=k^2\norm{e^{\mathcal{L}_\nu t}g}^2\leq k^2e^{-2\nu k^2t}\norm{g}^2,\\
\norm{\de_x(e^{\mathcal{L}_\nu t}g)}^2=k^2\norm{e^{\mathcal{L}_\nu t}g}^2\leq k^2e^{-\gamma |k|^\frac{1}{2}\nu^\frac{1}{2} t}\norm{g}^2.\\
\end{array}
\right.
\end{equation}
Using the fact that the function $x^ne^{-x}$ is bounded for $x\geq0$ for every $n\in \NN$,  we have
\begin{equation}
\left\{
\begin{array}{l}
k^2\norm{e^{\mathcal{L}_\nu t}g}^2\leq \frac{C}{\nu t}\norm{g}^2,\\
k^2\norm{e^{\mathcal{L}_\nu t}g}^2\leq \frac{C}{\nu^2t^4}\norm{g}^2\,\
\end{array}
\right.
\end{equation}
and hence 
\begin{equation}
\norm{\de_x(e^{\mathcal{L}_\nu t}g)}=|k|\norm{e^{\mathcal{L}_\nu t}g}\leq C \min\left\{\frac{1}{\sqrt{\nu t}},\frac{1}{\nu t^2}\right\}\norm{g}.
\end{equation}
We have that
\begin{align}
\min\left\{\frac{1}{\sqrt{\nu t}},\frac{1}{\nu t^2}\right\} =\left\{
\begin{aligned}
(\nu t)^{-\frac{1}{2}} \quad \text{for } t\leq \nu^{-\frac{1}{3}},\\
(\nu t^2)^{-1} \quad \text{for } t\geq \nu^{-\frac{1}{3}},\\
\end{aligned}
\right.
\end{align} 
and so
\begin{align*}
\int_0^T\norm{\de_x(e^{\mathcal{L}_\nu t}g)}&\leq C\left(\int_0^{\nu^{-\frac{1}{3}}}\frac{1}{\sqrt{\nu t}}\di t + \int_{\nu^{-\frac{1}{3}}}^T\frac{1}{\nu t^2}\di t\right)\norm{g}\\
&\leq C \left(\nu^{-\frac{1}{2}}\nu^{-\frac{1}{6}}+\nu^{-1}\nu^\frac{1}{3}\right)\norm{g}\\
&\leq C \nu^{-\frac{2}{3}}\norm{g}.
\end{align*}
For the third estimate, we apply the Minkowski inequality and we reduce to 
\begin{equation}\label{after minkowski}
\int_0^T\norm{\nabla\Delta^{-1}(e^{\mathcal{L}_\nu t}g)}^2_{L^\infty}\leq \left(\sum_{k>0} \left(\int_0^T\norm{\nabla_k\Delta_k^{-1}(e^{\mathcal{L}_\nu t}g)_k}_{L^\infty}^2\di t\right)^\frac{1}{2}\right)^2.
\end{equation}
Consider  $\norm{\nabla_k\Delta_k^{-1}(e^{\mathcal{L}_\nu t}g)_k}_{L^\infty}^2$ for $k\neq 0$.  Using the one dimensional Gagliardo-Niremberg-Sobolev inequality we have
\begin{equation}\label{prima stima grad psi}
\norm{\nabla_k\Delta_k^{-1}(e^{\mathcal{L}_\nu t}g)_k}_{L^\infty}^2\leq C \norm{\nabla_k\Delta_k^{-1}(e^{\mathcal{L}_\nu t}g)_k}\norm{(e^{\mathcal{L}_\nu t}g)_k}.
\end{equation}
Moreover, interpolating the $L^2$ norm we get
\begin{equation}\label{seconda stima grad psi}
\norm{\nabla_k\Delta_k^{-1}(e^{\mathcal{L}_\nu t}g)_k}_{L^\infty}^2\leq C  \norm{\nabla_k\Delta_k^{-1}(e^{\mathcal{L}_\nu t}g)_k}^\frac{3}{2}\norm{\nabla_k(e^{\mathcal{L}_\nu t}g)_k}^\frac{1}{2}.
\end{equation}
Combining \eqref{prima stima grad psi}  with \eqref{stima psi} we deduce 
\begin{equation}
\left\{
\begin{array}{l}
\norm{\nabla_k\Delta_k^{-1}(e^{\mathcal{L}_\nu t}g)_k}_{L^\infty}^2\leq (\nu k^2t^3)^{-\frac{1}{2}}\norm{g_k}^2,\\
\norm{\nabla_k\Delta_k^{-1}(e^{\mathcal{L}_\nu t}g)_k}_{L^\infty}^2\leq |k|^{-1}\norm{g_k}^2,\\
\end{array}
\right.
\end{equation}
where the second inequality follows from $|k|\norm{\nabla_k\Delta_k^{-1}(e^{\mathcal{L}_\nu t}g)_k}\leq\norm{(e^{\mathcal{L}_\nu t}g)_k}\leq\norm{g_k}$.
Hence, we deduce that
\begin{equation}\label{stima combinata psi}
\norm{\nabla_k\Delta_k^{-1}(e^{\mathcal{L}_\nu t}g)_k}_{L^\infty}^2\leq C|k|^{-1}\min\{1,(\nu t^3)^{-\frac{1}{2}}\}\norm{g_k}^2.
\end{equation}
As before,  using \eqref{stima combinata psi} we have
\begin{align*}
\int_0^T\norm{\nabla_k\Delta_k^{-1}(e^{\mathcal{L}_\nu t}g)_k}^2_{L^\infty}&\leq C|k|^{-1}\left( \int_0^T\min\{1,(\nu t^3)^{-\frac{1}{2}}\}\di t \right)\norm{g_k}^2\\
& \leq C |k|^{-1}\left(\nu^{-\frac{1}{3}}+ \int_{\nu^{-\frac{1}{3}}}^T(\nu t^3)^{-\frac{1}{2}}\di t\right)\norm{g_k}^2\\
&\leq C |k|^{-1} \left(\nu^{-\frac{1}{3}}+\nu^{-\frac{1}{2}}\nu^\frac{1}{6}\right)\norm{g_k}^2\\
&\leq C|k|^{-1} \nu^{-\frac{1}{3}}\norm{g_k}^2.
\end{align*}
Using \eqref{seconda stima grad psi} and the H\"older inequality we obtain
\begin{align*}
\int_0^T\norm{\nabla_k\Delta_k^{-1}(e^{\mathcal{L}_\nu t}g)_k}^2_{L^\infty}&\leq C\int_0^T \norm{\nabla_k\Delta_k^{-1}(e^{\mathcal{L}_\nu t}g)_k}^\frac{3}{2}\norm{\nabla_k(e^{\mathcal{L}_\nu t}g)_k}^\frac{1}{2}\\
& \leq C \left(\int_0^T \norm{\nabla_k\Delta_k^{-1}(e^{\mathcal{L}_\nu t}g)_k}^2\right)^\frac{3}{4} \left(\int_0^T\norm{\nabla_k(e^{\mathcal{L}_\nu t}g)_k}^2\right)^\frac{1}{4}\\
&\leq C \left(\int_0^T |k|^{-2}\min\{1,(\nu t^3)^{-1}\}\norm{g_k}^2\di t\right)^\frac{3}{4} \nu^{-\frac{1}{4}}\norm{g_k}^\frac{1}{2}\\
&\leq C|k|^{-\frac{3}{2}} \nu^{-\frac{1}{2}}\norm{g_k}^2.
\end{align*}
So, from \eqref{after minkowski} and using the Cauchy-Schwarz inequality we deduce 
\begin{align*}
\int_0^T\norm{\nabla\Delta^{-1}(e^{\mathcal{L}_\nu t}g)}_{L^\infty}^2&\leq C \left(\sum_{k>0}\left(\int_0^T\norm{\nabla_k\Delta_k^{-1}(e^{\mathcal{L}_\nu t}g)_k}_{L^\infty}^2\right)^\frac{1}{2}\right)^2 \\
&\leq C \left(\sum_{0<k<\nu^{-\frac{1}{3}}}\left(k^{-1}\nu^{-\frac{1}{3}}\norm{g_k}^2\right)^\frac{1}{2}+\sum_{k>\nu^{-\frac{1}{3}}}\left(k^{-\frac{3}{2}}\nu^{-\frac{1}{2}}\norm{g_k}^2\right)^\frac{1}{2}\right)^2\\
&\leq C \left( \nu^{-\frac{1}{3}}\sum_{0<k<\nu^{-\frac{1}{3}}}k^{-1}+\nu^{-\frac{1}{2}}\sum_{k>\nu^{-\frac{1}{3}}}k^{-\frac{3}{2}}\right)\sum_{k\neq0}\norm{g_k}^2\\
&\leq C  \left( \nu^{-\frac{1}{3}}|\log\nu|+\nu^{-\frac{1}{2}}\nu^\frac{1}{6}\right)\norm{g}^2\\
&\leq C  \nu^{-\frac{1}{3}}(1+|\log\nu|)\norm{g}^2.
\end{align*}
\end{proof}
\begin{rmk}
Regarding the third estimate, note that using only \eqref{prima stima grad psi} leads to the series $\sum_{k>0} \frac{1}{k}$, which is clearly not convergent.  The logarithmic correction term arises when we use \eqref{seconda stima grad psi} to overcome this problem.
\end{rmk}
\section{Nonlinear Transition Threshold}
Before proving the main result we start by considering the nonlinear problem \eqref{IVP Poiseuille}.  We decompose $\omega=\omega_s+\tilde{\omega}$, as well as $\psi=\psi_s+\tilde{\psi}$  and $u=u_s+\tilde{u}$, where $$\omega_s=\PP_0\omega=\frac{1}{2\pi}\int_{\TT}\omega\di x$$ 
is the shear part of the flow while
$$\quad \tilde{\omega}=\PP_{\neq}\omega=(1-\PP_0)\omega$$ is the non shear part.  It follows that $\omega_s$ and $\tilde{\omega}$ satisfy
\begin{align}
\left\{
\begin{aligned}
&\de_t\omega_s-\nu\de_y^2\omega_s=-\PP_0(\tilde{u}\cdot\nabla\tilde{\omega}),\\
&\de_t \tilde{\omega} + y^2\de_x\tilde{\omega}-2\de_x\tilde{\psi}-\nu\Delta\tilde{\omega}=-\PP_{\neq}(u_s\de_x\tilde{\omega}+\tilde{u}\cdot \nabla \omega).
\end{aligned}
\right.
\end{align}
Before proving Theorem \ref{transition thm}, we consider the following equation
\begin{equation}
\de_t \tilde{\omega} -\mathcal{L}_\nu \tilde{\omega}=f,
\end{equation}
where $\mathcal{L}_\nu$ is the linearized operator defined in \eqref{linear operator} and $f$ is a time-dependent function.  Using the Duhamel formula we can write the solution for every $t\geq 0$ as 
\begin{equation}
\tilde{\omega}(t)=e^{t\mathcal{L}_\nu}\tilde{\omega}(0)+e^{t\mathcal{L}_\nu}\int_0^te^{-s\mathcal{L}_\nu}f(s)\di s.
\end{equation} 
Using Lemma \ref{main lemma} we obtain similar estimates for the solution $\tilde{\omega}(t)$ for any arbitrary $T>0$.  We consider firstly 
\begin{equation}
\int_0^T\norm{\de_x\tilde{\omega}}\di t\leq \int_0^T\norm{\de_xe^{t\mathcal{L}_\nu}\tilde{\omega}(0)}\di t+\int_0^T\norm{\de_xe^{t\mathcal{L}_\nu}\int_0^te^{-s\mathcal{L}_\nu}f(s)\di s}\di t.
\end{equation}
The first term on the right hand side can be bounded using \eqref{secondastima} as 
\begin{equation}
\int_0^T\norm{\de_xe^{t\mathcal{L}_\nu}\tilde{\omega}(0)}\di t\leq C\nu^{-2/3}\norm{\tilde{\omega}(0)}
\end{equation}
while for the second term we use the Minkowski integral inequality and estimate \eqref{secondastima} again as follows
\begin{align}
\begin{aligned}
\int_0^T\norm{\de_xe^{t\mathcal{L}_\nu}\int_0^te^{-s\mathcal{L}_\nu}f(s)\di s}\di t &= \int_0^T\norm{\int_0^t\de_xe^{(t-s)\mathcal{L}_\nu}f(s)\di s}\di t \\
&\leq \int_0^T\int_0^T\norm{\de_xe^{(t-s)\mathcal{L}_\nu}f(s)}\di t\di s \\
&\leq C\nu^{-2/3}\int_0^T \norm{f(s)}\di s.
\end{aligned}
\end{align}
Combining the two estimates we have
\begin{equation}
\nu^{2/3}\int_0^T\norm{\de_x\tilde{\omega}}\leq C\left( \norm{\tilde{\omega}(0)}+\int_0^T\norm{f(s)}\di s\right).
\end{equation}
Analogously, using estimate \eqref{terzastima} we have 
\begin{align}
\begin{aligned}
\left(\int_0^T\norm{\nabla\Delta^{-1}\tilde{\omega}(t)}_{L^\infty}^2\di t\right)^{1/2} & \leq \left(\int_0^T\norm{\nabla\Delta^{-1} e^{t\mathcal{L}_\nu}\tilde{\omega}(0)}^2\di t\right)^{1/2}\\
&\quad +\left(\int_0^T\norm{\nabla\Delta^{-1} e^{t\mathcal{L}_\nu}\int_0^te^{-s\mathcal{L}_\nu}f(s)\di s}\di t\right)^{1/2}\\
&\leq C(|\log\nu|+1)^{1/2}\nu^{-1/6}\left( \norm{\tilde{\omega}(0)}+\int_0^T\norm{f(s)}\di s\right)
\end{aligned}
\end{align}
and we deduce 
\begin{align}\label{ineq nonlinear}
\begin{aligned}
\nu^{2/3}\int_0^T\norm{\de_x\tilde{\omega}}\di t+(|\log\nu|+1)^{-1/2}\nu^{1/6}\left(\int_0^T\norm{\nabla\Delta^{-1}\tilde{\omega}(t)}_{L^\infty}^2\di t\right)^{1/2}&\\ \leq C \norm{\tilde{\omega}(0)}+&C\int_0^T\norm{f(s)}\di s.
\end{aligned}
\end{align}
We are ready to prove the nonlinear threshold theorem.
\begin{theorem}\label{transition thm}
There exists constants $\varepsilon_0\in (0,1), C_1>0,c_1>0$ such that for all $\nu<1$ for every $\omega_{in}\in L^2$ with $$\norm{\omega_{in}}\leq \varepsilon_0(1+|\log\nu|^\frac{1}{2})^{-1}\nu^{2/3},$$ the solution of \eqref{IVP Poiseuille} is global in time with the bound $$\norm{\tilde{\omega}} \leq C_1e^{-c_1\nu^{1/2}t}\norm{\tilde{\omega}_{in}}.$$
\end{theorem}
\begin{proof}
Denote by $\mathcal{L}_\nu$ the linearized operator,  we have 
\begin{equation}
\de_t\tilde{\omega} -\mathcal{L}_\nu\tilde{\omega}=-\PP_{\neq}(u_s\de_x\tilde{\omega}+\tilde{u}\cdot \nabla \omega),
\end{equation}
thus, applying Duhamel formula we obtain
\begin{equation}
\tilde{\omega}(t)=e^{t\mathcal{L}_\nu}\tilde{\omega}_{in}-e^{t\mathcal{L}_\nu}\int_0^te^{-s\mathcal{L}_\nu}\PP_{\neq}(u_s\de_x\tilde{\omega}+\tilde{u}\cdot \nabla\omega)\di s.
\end{equation}
Taking the $L^2$ norm of the solution and using the linear enhanced dissipation estimates \eqref{enhanced diss} of the solution operator $e^{\mathcal{L}_\nu t}$ leads to
\begin{align}\label{stima tilde}
\begin{aligned}
\norm{\tilde{\omega}(t)}&\leq C_0e^{-c_0\nu^\frac{1}{2}t}\norm{\tilde{\omega}_{in}}+\int_0^t\norm{u_s\de_x\tilde{\omega}+\tilde{u}\cdot \nabla\omega}\di s\\
&\leq C_0e^{-c\nu^\frac{1}{2}t}\norm{\tilde{\omega}_{in}}+\int_0^t\norm{u_s\de_x\tilde{\omega}}\di s+\int_0^t\norm{\tilde{u}\cdot \nabla\omega}\di s,
\end{aligned}
\end{align}
where $c_0, C_0$ are the constants defined in Theorem \ref{thm hypo intro}.
Defining 
\begin{equation}\label{def A(t)}
A(t)=\nu^\frac{2}{3}\int_0^t\norm{\de_x\tilde{\omega}}+\nu^\frac{1}{6}(|\log\nu|+1)^{-\frac{1}{2}}\left(\int_0^t\norm{\tilde{u}}_{L^\infty}^2\di s\right)^\frac{1}{2},
\end{equation}
it follows from \eqref{ineq nonlinear} that
\begin{equation}\label{stima A(t)}
 A(t) \leq K\left(\norm{\tilde{\omega}_{in}}+\int_0^t\norm{u_s\de_x\tilde{\omega}}\di s+\int_0^t\norm{\tilde{u}\cdot \nabla\omega}\di s\right),
\end{equation}
for a fixed positive constant $K.$
To achieve nonlinear enhanced dissipation we proceed as follows.
For the second term on the right hand side in \eqref{stima tilde}, we have by using the definition of $A(t)$ and the hypothesis
\begin{align}\label{stima secondo}
\begin{aligned}
\int_0^t\norm{u_s\de_x\tilde{\omega}}\di s & \leq \int_0^t\norm{u_s}_{L^\infty}\norm{\de_x\tilde{\omega}}\di s\\
&\leq C\norm{\omega_{in}}\int_0^t\norm{\de_x\tilde{\omega}}\di s \\
&\leq C\nu^{-\frac{2}{3}}\norm{\omega_{in}}A(t)\\
&\leq \frac{1}{8K} A(t),
\end{aligned}
\end{align}
where in the last inequality we use that $\varepsilon_0\ll1$.
Analogously, for the third term on the right hand side in \eqref{stima tilde}, we have
\begin{align}\label{stima terzo}
\begin{aligned}
\int_0^t\norm{\tilde{u}\cdot \nabla\omega}&\leq \int_0^t\norm{\tilde{u}}_{L^\infty}\norm{\nabla\omega}\\
&\leq \left(\int_0^t\norm{\tilde{u}}_{L^\infty}^2\right)^\frac{1}{2}\left(\int_0^t\norm{\nabla\omega}^2\right)^\frac{1}{2}\\
&\leq\left( \nu^{-\frac{1}{6}}(|\log\nu|+1)^{\frac{1}{2}}A(t)\right)\left(\nu^{-\frac{1}{2}}\norm{\omega_{in}}\right)\\
&\leq \nu^{-\frac{2}{3}}(|\log\nu|+1)^{\frac{1}{2}}\norm{\omega_{in}}A(t)\\
&\leq \frac{1}{8K}A(t).
\end{aligned}
\end{align}
From \eqref{stima A(t)}, \eqref{stima secondo} and \eqref{stima terzo}, we can deduce that 
\begin{equation}
A(t)\leq  K\norm{\tilde{\omega}_{in}}+\frac{1}{4}A(t),
\end{equation}
hence 
\begin{equation}\label{stima su At}
A(t)\leq \frac{4}{3} K \norm{\tilde{\omega}_{in}}.
 \end{equation} 
Moreover we have 
\begin{equation}
\int_0^t\norm{u_s\de_x\tilde{\omega}+\tilde{u}\cdot \nabla\omega}\di s\leq \int_0^t\norm{u_s\de_x\tilde{\omega}}\di s+\int_0^t\norm{\tilde{u}\cdot \nabla\omega}\di s \leq \frac{1}{4K} A(t)\leq \frac{1}{3}\norm{\tilde{\omega}_{in}}.
\end{equation}
Fix now $t_0$ such that $C_0e^{-c_0\nu^\frac{1}{2}t_0}\leq\frac{1}{6}$, we have that 
\begin{equation}
\norm{\tilde{\omega}(t_0)}\leq \frac{1}{6}\norm{\tilde{\omega}_{in}}+\frac{1}{3}\norm{\tilde{\omega}_{in}}= \frac{1}{2}\norm{\tilde{\omega}_{in}}
\end{equation}
and 
\begin{equation}
\norm{\tilde{\omega}(t)}\leq C\norm{\tilde{\omega}_{in}},
\end{equation}
for all $0<t<t_0$.
To conclude the proof we use the same iteration argument used for the linear enhanced dissipation.  To do so, we prove that for every $s\geq 0$ it holds
\begin{equation}\label{est1}
\norm{\tilde{\omega}(s+t_0)}\leq \frac{1}{2}\norm{\tilde{\omega}(s)}
\end{equation}
and
\begin{equation}\label{est2}
\norm{\tilde{\omega}(s+t^*)}\leq C\norm{\tilde{\omega}(s)},
\end{equation}
for any $t^*\in(0,t_0].$
Hence, by iteration,  it follows that 
\begin{equation}
\norm{\tilde{\omega}(t)}\leq C\norm{\tilde{\omega}(\lfloor t_0^{-1}t\rfloor t_0)} \leq C\left(\frac{1}{2}\right)^{\lfloor t_0^{-1}t\rfloor}\norm{\tilde{\omega}_{in}}\leq C_1e^{-c_1\nu^\frac{1}{2}t}\norm{\tilde{\omega}_{in}}
\end{equation}
for all $t>0$, and the theorem is proved.
Estimates \eqref{est1} and \eqref{est2} follow from the previous computations done for $s=0$ and the fact that  $\norm{\omega_s(t)}\leq\norm{\omega_{in}}$.  Indeed, we have
\begin{equation}
\tilde{\omega}(s+t^*)=e^{t^*\mathcal{L}_\nu}\tilde{\omega}(s)-e^{t^*\mathcal{L}_\nu}\int_0^{t^*}e^{-\tau\mathcal{L}_\nu}\PP_{\neq}(u_s\de_x\tilde{\omega}+\tilde{u}\cdot \nabla\omega)(\tau+s)\di \tau
\end{equation}
and hence, analogously to \eqref{stima tilde}
\begin{equation}
\norm{\tilde{\omega}(s+t^*)}\leq e^{-c\nu^\frac{1}{2}t^*}\norm{\tilde{\omega}(s)}+\int_0^{t^*}\norm{u_s\de_x\tilde{\omega}(\tau +s)}\di \tau+\int_0^{t^*}\norm{\tilde{u}\cdot \nabla\omega(\tau+s)}\di \tau.
\end{equation}
In order to proceed with the same argument, we notice that 
\begin{equation}
\norm{u_s(\tau+s)}_{L^\infty}^2\leq C \norm{u_s(\tau+s)}\norm{\omega_s(\tau+s)}\leq C \norm{\omega_s(\tau+s)}^2\leq C\norm{\omega_{in}}^2
\end{equation}
where the last inequality follow from the fact that $t\mapsto \norm{\omega(t)}$ is decreasing.
Moreover we have
\begin{equation}
\norm{\omega(s+t^*)}^2+\frac{\nu}{2}\int_0^{t^*}\norm{\nabla\omega(\tau+s)}^2\di\tau\leq \norm{\omega(s)}^2\leq \norm{\omega_{in}}^2.
\end{equation}
Hence,   defining, as done in \eqref{def A(t)}, 
\begin{equation}
B(t^*;s)=\nu^\frac{2}{3}\int_0^{t^*}\norm{\de_x\tilde{\omega}(\tau+s)}\di \tau+\nu^\frac{1}{6}(|\log\nu|+1)^{-\frac{1}{2}}\left(\int_0^{t^*}\norm{\tilde{u}(\tau+s)}_{L^\infty}^2\di \tau\right)^\frac{1}{2},
\end{equation}
we can deduce the analogous estimate \eqref{stima su At}, 
\begin{equation}
B(t^*; s)\leq \frac{4}{3}K\norm{\tilde{\omega}(s)}
\end{equation}
and finally
\begin{equation}
\norm{\tilde{\omega}(s+t^*)}\leq C_0e^{-c_0\nu^{1/2}t^*}\norm{\tilde{\omega}(s)}+\frac{1}{3}\norm{\tilde{\omega}(s)}
\end{equation}
from which \eqref{est1}, \eqref{est2} follow.
\end{proof}

\paragraph{Acknowledgments}
The author wants to thank Michele Coti Zelati for his support and all the useful  and stimulating discussions about this problem.

\bibliographystyle{abbrv}
\bibliography{Biblio.bib}

\end{document}